\documentclass[12pt,oneside]{amsart}

\usepackage{amsthm,amsmath,amssymb}
\usepackage{mathrsfs}
\usepackage{enumerate}
\usepackage{amsfonts}
\usepackage{verbatim}
\usepackage{cite}
\usepackage{amsbsy}
\usepackage{amsmath}
\usepackage{amssymb}
\usepackage{MnSymbol}
\usepackage{mathrsfs}

\newtheorem{theorem}{Theorem}[section]
\newtheorem{lemma}[theorem]{Lemma}
\newtheorem{proposition}[theorem]{Proposition}

\newtheorem*{claim}{Claim}
\newtheorem{question}[theorem]{Question}
\newtheorem{corollary}[theorem]{Corollary} 
\theoremstyle{definition}
\newtheorem{definition}[theorem]{Definition}

\theoremstyle{remark}
\newtheorem{remark}[theorem]{Remark}
\newtheorem{example}[theorem]{Example}

\DeclareMathOperator{\supp}{supp}

\DeclareMathOperator{\Sym}{Sym}


\begin{document}

\title{Parikh Matrices and Strong $M$-Equivalence}
\author{Wen Chean Teh}
\address{School of Mathematical Sciences\\
Universiti Sains Malaysia\\
11800 USM, Malaysia}
\email{dasmenteh@usm.my}

\maketitle

\begin{abstract}
Parikh matrices have been a powerful tool in arithmetizing words by numerical quantities. However, the dependence on the ordering of the alphabet is inherited by Parikh matrices. Strong $M$-equivalence is proposed as a canonical alternative to $M$-equivalence to get rid of this undesirable property. Some characterization of strong $M$-equivalence for a restricted class of words is obtained. Finally, the existential counterpart of strong $M$-equivalence is introduced as well.
\end{abstract}

\section{introduction}
Parikh matrices were introduced in \cite{MSSY01} as an extension of the Parikh vectors \cite{rP66}. 
The definition of Parikh matrices is ingenious, natural, intuitive and amazingly simple. Still, Parikh matrices prove to be a powerful tool in studying (scattered) subword occurrences, for example, see \cite{DS06,MSY04,aS05b,aS06,aS08}.
Nevertheless, due to the limited number of entries in a Parikh matrix, not every word is uniquely determined by its Parikh matrix.
Two words are $M$-equivalent if{f} they have the same Parikh matrix and a word is $M$-unambiguous if{f} it is not $M$-equivalent to another distinct word.
The characterization of $M$-equivalence, as well as, $M$-ambiguity has been the most actively researched problem in this area, for example, see \cite{aA07,AAP08,aA14,AMM02,FR04,aS05a,aS10,vS09,SS06,wT14,wT15a,TK14}. 

Inherent in the definition of Parikh matrices is the dependency on the ordering of the alphabet. Because of this, the words $acb$ and $cab$ are $M$-equivalent with respect to $\{a<b<c\}$ but they are not $M$-equivalent with respect to $\{a<c<b\}$. This undesirable property has led us to propose the notion of strong $M$-equivalence, one that is absolute in the sense that not only dependency on the ordering of the alphabet is avoided, independency from the alphabet is also achieved.  This notion is in fact a natural combinatorial property between words that does not rely on Parikh matrices. 

The remainder of this paper is structured as follows.  Section~2 provides the basic definitions and terminology.
The next section introduces the notion of strong $M$-equivalence and a strictly weaker notion of $M\!S\!E$-equivalence.
However, $M\!S\!E$-equivalence is shown to coincide with strong $M$-equivalence for sufficiently simple ternary words in the subsequent section.
As a side result, relatively simple counterexamples witnessing the fact that $M\!E$-equivalence is strictly weaker than $M$-equivalence are shown to exist,  as opposed  to the mysterious counterexample of length $15$ provided in \cite{vS09}. 
The following section sees the introduction of the opposite notion of ``weakly $M$-related". Our conclusions follow after that, highlighting some future problems.

\section{Subwords and Parikh Matrices}

The cardinality of a set $X$ is denoted by $\vert X\vert$.

Suppose $\Sigma$ is a finite alphabet. The set of words over $\Sigma$ is denoted by $\Sigma^*$. The empty word is denoted by $\lambda$. 
Let $\Sigma^+$ denote the set $\Sigma^*\backslash\{\lambda\}$. If $v,w\in \Sigma^*$, the concatenation of $v$ and $w$ is denoted by $vw$.
An \emph{ordered alphabet} is an alphabet $\Sigma= \{a_1, a_2, \dotsc,a_s\}$ with an ordering on it. For example, if $a_1<a_2<\dotsb < a_s$, then we may write
$\Sigma= \{a_1<a_2< \dotsb<a_s\}$. On the other hand, if $ \Sigma=\{a_1< a_2< \dotsb< a_s\}  $ is an ordered alphabet, then the \emph{underlying alphabet} is  $\{a_1, a_2, \dotsc,a_s\}$.
For $1\leq i\leq j \leq s$, let $a_{i,j}$ denote the word $a_ia_{i+1}\dotsm a_j$.
Frequently, we will abuse notation and use $\Sigma$ to stand for both the ordered alphabet and its underlying alphabet, for example, as in ``$w\in \Sigma^*$", when $\Sigma$ is an ordered alphabet.
If $w\in \Sigma^*$, then $\vert w\vert$ is the length of $w$. 
Suppose $\Gamma\subseteq \Sigma$. The projective morphism $\pi_{\Gamma}\colon \Sigma^*\rightarrow \Gamma^*$ is defined by
$$\pi_{\Gamma}(a)=\begin{cases}
a, & \text{if } a\in \Gamma\\
\lambda, & \text{otherwise.}
\end{cases}$$
We may write $\pi_{a,b}$ for $\pi_{\{a,b\}}$.

\begin{definition}
A word $w'$ is a \emph{subword} of a word $w\in \Sigma^*$ if{f} there exist $x_1,x_2,\dotsc, x_n, y_0, y_1, \dotsc,y_n\in \Sigma^*$, some of them possibly empty, such that
$$w'=x_1x_2\dotsm x_n \text{ and } w=y_0x_1y_1\dotsm y_{n-1}x_ny_n.$$
\end{definition}

In the literature, our subwords are usually called ``scattered subwords". The number of occurrences of a word $u$ as a subword of $w$ is denoted by $\vert w\vert_u$. Two occurrences of $u$ are considered different if{f} they differ by at least one position of some letter. For example, $\vert aabab\vert_{ab}=5$ and $\vert baacbc\vert_{abc}=2$. By convention, $\vert w\vert_{\lambda}=1$ for all $w\in \Sigma^*$. The \emph{support} of $w$, denoted $\supp(w)$, is the set $\{ a\in \Sigma\mid \vert w\vert_a \neq 0   \}$. Note that the support of $w$ is independent of $\Sigma$. The reader is referred to \cite{RS97} for language theoretic notions not detailed here.

For any integer $k\geq 2$, let $\mathcal{M}_k$ denote the multiplicative monoid of $k \times k$ upper triangular matrices with nonnegative integral entries and unit diagonal.

\begin{definition} 
Suppose $\Sigma=\{a_1<a_2< \dotsb<a_s\}$ is an ordered alphabet. The \emph{Parikh matrix mapping}, denoted $\Psi_{\Sigma}$, is the monoid morphism
$$ \Psi_{\Sigma}\colon \Sigma^*\rightarrow \mathcal{M}_{s+1}$$
defined as follows:\\
if $\Psi_{\Sigma}(a_q)=(m_{i,j})_{1\leq i,j\leq s+1}$, then $m_{i,i}=1$ for each $1\leq i\leq s+1$, $m_{q,q+1}=1$ and all other entries of the matrix $\Psi_{\Sigma}(a_q)$ are zero. 
Matrices of the form  $\Psi_{\Sigma}(w)$ for $w\in \Sigma^*$ are called \emph{Parikh matrices}.
\end{definition}


\begin{theorem}\label{1206a}\cite{MSSY01}
Suppose $\Sigma=\{a_1<a_2< \dotsb<a_s\}$ is an ordered alphabet and $w\in \Sigma^*$. The matrix $\Psi_{\Sigma}(w)=(m_{i,j})_{1\leq i,j\leq s+1}$ has the following properties:
\begin{itemize}
\item $m_{i,i}=1$ for each $1\leq i \leq s+1$;
\item $m_{i,j}=0$ for each $1\leq j<i\leq s+1$;
\item $m_{i,j+1}=\vert w \vert_{a_{i,j}}$ for each $1\leq i\leq j \leq s$.
\end{itemize}
\end{theorem}

The \emph{Parikh vector} $\Psi(w)=(\vert w\vert_{a_1}, \vert w\vert_{a_2}, \dotsc, \vert w\vert_{a_s})$ of a word $w\in \Sigma^*$ is embedded in the second diagonal of the Parikh matrix $\Psi_{\Sigma}(w)$.

\begin{example}
Suppose $\Sigma=\{a<b<c\}$ and $w=babcc$.
Then \begin{align*}
\Psi_{\Sigma}(w)&=\Psi_{\Sigma}(b)\Psi_{\Sigma}(a)\Psi_{\Sigma}(b)\Psi_{\Sigma}(c)\Psi_{\Sigma}(c)\\
&= \begin{pmatrix}
1 & 0& 0&0 \\
0 &1 & 1 & 0\\
0 & 0 & 1 & 0\\
0 & 0 & 0 & 1
\end{pmatrix}
\begin{pmatrix}
1 & 1 & 0&0 \\
0 &1 & 0 & 0\\
0 & 0 & 1 & 0\\
0 & 0 & 0 & 1
\end{pmatrix}\dotsm
\begin{pmatrix}
1 & 0 & 0&0 \\
0 &1 & 0 & 0\\
0 & 0 & 1 & 1\\
0 & 0 & 0 & 1
\end{pmatrix}\\
&= \begin{pmatrix}
1 & 1& 1& 2 \\
0 &1 & 2 & 4\\
0 & 0 & 1 & 2\\
0 & 0 & 0 & 1
\end{pmatrix}
=\begin{pmatrix}
1 & \vert w\vert_a &\vert w\vert_{ab} &\vert w\vert_{abc} \\
0 &1 & \vert w\vert_b & \vert w\vert_{bc}\\
0 & 0 & 1 & \vert w\vert_c\\
0 & 0 & 0 & 1
\end{pmatrix}.
\end{align*}
\end{example}


\begin{definition}
Suppose $\Sigma=\{a_1<a_2< \dotsb<a_s\}$  is an ordered alphabet.
\begin{enumerate}
\item Two words $w,w'\in \Sigma^*$ are \emph{$M$-equivalent}, denoted $w\equiv_M w'$, if{f} $\Psi_{\Sigma}(w)=\Psi_{\Sigma}(w')$.
\item A word $w\in \Sigma^*$ is \emph{$M$-unambiguous} if{f} no distinct word is $M$-equivalent to $w$. Otherwise, $w$ is said to be \emph{$M$-ambiguous}.
\end{enumerate}
\end{definition}

Note that the notion of $M$-equivalence, as well as $M$-ambiguity, depends on the ordered alphabet $\Sigma$.
However, for the various relations that we will encounter in this article, the reference to the respective ordered/unordered alphabet often will be suppressed, assuming it is understood from the context.

The following are two elementary rules for deciding whether two words are $M$-equivalent.
Suppose $\Sigma=\{a_1<a_2< \dotsb<a_s\}$ and $w,w'\in \Sigma^*$. 
\begin{itemize}
\item[$E1$.] If $w=xa_ka_ly$ and $w'=xa_la_k y$ for some $x, y\in \Sigma^*$ and $\vert k-l\vert \geq 2$, 
then $w\equiv_M w'$.
\item[$E2$.] If $w=xa_ka_{k+1}ya_{k+1}a_kz$ and $w'=xa_{k+1}a_kya_ka_{k+1}z$ for some $1\leq k\leq s-1$, $x, z\in \Sigma^*$, and  $y\in (\Sigma\backslash \{a_{k-1},a_{k+2}\})^*$, 
then $w\equiv_M w'$.
\end{itemize}

Rule $E1$ is obviously valid by Theorem~\ref{1206a}. In other words, the Parikh matrix of a word $w$ is not sensitive to the mutual ordering of any two consecutive distinct letters in $w$ that are not consecutive in the ordered alphabet.
Meanwhile, Rule $E2$ is sufficient to characterize $M$-equivalence for the binary alphabet.

\begin{theorem}\cite{aA07,FR04}\label{2605c}
Suppose $\Sigma$ is a binary ordered alphabet and $w,w'\in \Sigma^*$. Then $w$ and $w'$ are $M$-equivalent if and only if $w'$ can be obtained from $w$ by finitely many applications of Rule $E2$ (more precisely, the rewriting rule implicitly stated in Rule~$E2$).
\end{theorem}

\begin{example}
Applying Rule $E2$, $w=\boldsymbol{ba}a\boldsymbol{ab}bba \rightarrow ab\boldsymbol{ab}ab\boldsymbol{ba} \rightarrow abbaabab=w'$. Hence,  $w$ and $w'$ are $M$-equivalent with respect to $\{a<b\}$. 
\end{example}

\begin{definition}\label{1906a}
Suppose  $\Sigma$ is an ordered alphabet and $w,w'\in \Sigma^*$. 
We say that $w$ and $w'$ are \emph{$1$-equivalent}, denoted $w\equiv_1 w'$, if{f} $w'$ can be obtained from $w$ by finitely many applications of Rule $E1$. 
We say that $w$ and $w'$ are \emph{elementarily matrix equivalent} (\emph{$M\!E$-equivalent}), denoted $w\equiv_{M\!E} w'$, if{f} $w'$ can be obtained from $w$ by finitely many applications of Rule $E1$ and Rule $E2$. 
\end{definition}

The term $M\!E$-equivalence is due to Salomaa \cite{aS10}. Historically, it was claimed in \cite{AAP08} that $M\!E$-equivalence characterizes \mbox{$M$-equivalence}  for any alphabet.
However, it was overturned in \cite{vS09} through the counterexample $babcbabcbabcbab$ and $bbacabbcabbcbba$. The two words are $M$-equivalent with respect to $\{a<b<c\}$ but neither Rule $E1$ nor Rule $E2$ can be applied to the former word.

Finally, a simple lemma relating Parikh matrices to morphisms induced by permutations on the ordered alphabet is needed.
Suppose $\Sigma=\{a_1<a_2<\dotsb<a_s\}$ is an ordered alphabet.
Let $\Sym (s)$ denote the symmetric group of order $s$ and $\sigma \in \Sym(s)$. Then $\sigma$ induces a morphism $\sigma'$ from $\Sigma^*$ onto $\Sigma^*$ defined by
$\sigma'(a_i)=a_{\sigma(i)}$ for $1\leq i\leq s$. For simplicity, we may identify $\sigma'$ with $\sigma$ and write $\sigma w$ for $\sigma'(w)$.
Let $\sigma\Sigma$ denote the ordered alphabet $\{a_{\sigma^{-1}(1)}<a_{\sigma^{-1}(2)}  <\dotsb< a_{\sigma^{-1}(s)}   \}$. 

\begin{lemma}\label{2405c}
Suppose $\Sigma=\{a_1<a_2<\dotsb<a_s\}$ and $\sigma \in \Sym(s)$.
Then $\Psi_{\Sigma}(\sigma w)=\Psi_{\sigma\Sigma} (w)$ for all $w\in \Sigma^*$.
\end{lemma}

\begin{proof}
Since $\Psi_{\Sigma}$, $\Psi_{\sigma\Sigma}$ and $\sigma$ are morphisms, the lemma holds since for $1\leq q\leq s$,
$\Psi_{\Sigma}(\sigma a_q)=\Psi_{\sigma\Sigma} (a_q)$, which follows because $\sigma a_q=a_{\sigma (q)}$ and $a_q=a_{\sigma^{-1}(\sigma (q))}$.
\end{proof}

\section{Strong $M$-equivalence}\label{0206d}

The core object of this study will now be formally introduced.  As a reminder, unless explicitly stated, an alphabet does not come with an ordering on it. 

\begin{definition}
Suppose $\Sigma$ is an alphabet.
Two words $w, w'\in \Sigma^*$ are \emph{strongly $M$-equivalent}, denoted $w \overset{s}{\equiv}_M w'$, if{f} $w$ and $w'$ are $M$-equivalent  with respect to any ordered alphabet with underlying alphabet $\Sigma$.
\end{definition}


\begin{proposition}
Suppose $\Sigma$ is an alphabet. The relation $\overset{s}{\equiv}_M$ is an equivalence relation on $\Sigma^*$ that is left invariant  (respectively right invariant), meaning
$w \overset{s}{\equiv}_M w'$ if and only if $vw \overset{s}{\equiv}_M vw'$ (respectively $wv \overset{s}{\equiv}_M w'v$) for all $w,w',v\in \Sigma^*$. 
\end{proposition}

\begin{proof}
This follows because $M$-equivalence has such properties.
\end{proof}

Generally, strong $M$-equivalence is strictly stronger than $M$-equivalence. However, if $w$ and $w'$ are $M$-equivalent with respect to $\{a<b\}$, then 
$w$ and $w'$ are $M$-equivalent with respect to $\{b<a\}$ as well by Theorem~\ref{1206a} and the identity $\vert v\vert_{ab}+\vert v \vert_{ba}= \vert v\vert_a\vert v \vert_b$.
Hence, for the binary alphabet, strong $M$-equivalence is nothing more than $M$-equivalence.

\begin{theorem}\label{2701b}
Suppose $\Sigma$ is an alphabet and $w, w'\in \Sigma^*$. Then $w$ and $w'$ are strongly $M$-equivalent if and only if for every $v\in \Sigma^*$ such that
$\vert v\vert_a \leq 1$ for all $a\in \Sigma$, we have $\vert w\vert_v=\vert w'\vert_v$.
\end{theorem}

\begin{proof}
This is immediate from Theorem~\ref{1206a} and the definition.
\end{proof}



Theorem~\ref{2701b} should put to rest any doubt on the significance of strong \mbox{$M$-equivalence.} It shows that strong $M$-equivalence is  indeed a very natural and symmetrical combinatorial property between words. Futhermore, the property does not invoke Parikh matrices and thus need not factor in the ordering of the alphabet. In fact, this could have taken to be the defining property of strong $M$-equivalence. 

Furthermore, unlike $M$-equivalence, strong $M$-equivalence is \emph{absolute} in the sense that it is independent from the alphabet, as provided by the next proposition.
Hence, we can simply say that two words are strongly $M$-equivalent without implicitly/explicitly referring to any specific alphabet.

\begin{proposition}\label{1605c}
Suppose $\Sigma$ and $\Gamma$ are alphabets and $w,w'\in \left(\Sigma\cap \Gamma\right)^*$. Then $w$ and $w'$ are strongly $M$-equivalent with respect to $\Sigma$ if and only if they are strongly $M$-equivalent with respect to $\Gamma$. 
\end{proposition}

\begin{proof}
It suffices to prove any one direction. Assume $w$ is strongly $M$-equivalent to $w'$ with respect to $\Sigma$. 
Fix $v\in \Gamma^*$ such that
$\vert v\vert_a \leq 1$ for all $a\in \Gamma$.  If $v\notin  \Sigma^*$, then clearly 
$\vert w\vert_v=0=\vert w'\vert_v$. Otherwise, $\vert w\vert_v=\vert w'\vert_v$ by Theorem~\ref{2701b}. 
Therefore, $w$ and $w'$ are strongly $M$-equivalent with respect to $\Gamma$ by Theorem~\ref{2701b} again.
\end{proof}

From the definition, two words are strongly $M$-equivalent if and only if they are indistinguishable by any Parikh matrix mapping with respect to some ordering on the alphabet.
However, we can still cast strong $M$-equivalence in terms of a single Parikh matrix mapping.

\begin{proposition}\label{2106a}
Suppose $\Sigma$ is an ordered alphabet and $w, w'\in \Sigma^*$. 
\begin{enumerate}
\item Then $w$ and $w'$ are strongly $M$-equivalent if and only if $\Psi_{\Sigma}(\sigma w)=\Psi_{\Sigma}(\sigma w')$ for all  $\sigma\in \Sym (\vert \Sigma\vert)$.
\item If $w$ and $w'$ are strongly $M$-equivalent, then 
$\sigma w$ and $\sigma w'$ are strongly $M$-equivalent for all $\sigma\in \Sym (\vert \Sigma\vert)$.
\item If $w$ and $w'$ are strongly $M$-equivalent, then $\pi_{\Gamma}(w)$ and $\pi_{\Gamma}(w')$ are strongly $M$-equivalent for all $\Gamma \subseteq \Sigma$.
\end{enumerate}
\end{proposition}

\begin{proof}
Part $1$ follows from Lemma~\ref{2405c} while Part $2$ follows from Part $1$. 
Part~$3$ follows from Theorem~\ref{2701b} because $\vert \pi_{\Gamma}(w)\vert_v=\vert w\vert_v$ for all $w\in \Sigma^*$ and $v\in \Gamma^*$.
\end{proof}

Clearly, if $w'$ is obtained from $w$ by swapping some two consecutive distinct letters, then  $w$ and $w'$ are  not strongly $M$-equivalent.
The following rule is an analogue of Rule~$E2$ for strong $M$-equivalence. The next proposition shows that it is  a sound rule.
Suppose $\Sigma$ is an alphabet and $w,w'\in \Sigma^*$. 
\begin{itemize}
\item[$S\!E$.] If $w=x\alpha\beta y\beta \alpha z$ and $w'=x\beta \alpha y\alpha\beta z$ for some $\alpha,\beta\in \Sigma$, $x,z\in \Sigma^*$, and $y\in \{\alpha,\beta\}^*$,  then $w$ and $w'$ are strongly $M$-equivalent.
\end{itemize}

\begin{remark}\label{2106c}
For the binary alphabet, Rule~$S\!E$ coincides with Rule~$E2$.  
\end{remark}

\begin{definition}
Suppose  $\Sigma$ is an alphabet and $w,w'\in \Sigma^*$. 
We say that $w$ and $w'$ are \emph{strongly elementarily matrix equivalent} (\emph{$M\!S\!E$-equivalent}), denoted $w\equiv_{M\!S\!E} w'$, if{f} $w'$ results from $w$ by finitely many applications of Rule $S\!E$. 
\end{definition}

\begin{proposition}\label{2905b}
Suppose $\Sigma$ is an alphabet and $w, w'\in \Sigma^*$. If $w$ and $w'$ are $M\!S\!E$-equivalent, then they are strongly $M$-equivalent.
\end{proposition}

\begin{proof}
We may assume that $w'$ is obtained from $w$ by a single application of rule $S\!E$ because of transitivity of strong $M$-equivalence. Suppose $w=x \alpha\beta y\beta \alpha z$ and $w'=x\beta \alpha y\alpha\beta z$ for some $\alpha,\beta\in \Sigma$, $x,z\in \Sigma^*$, and $y\in \{\alpha,\beta\}^*$. 
Suppose $\Gamma$ is any ordered alphabet with underlying alphabet $\Sigma$. By Rule $E1$ and Rule $E2$, $w$ is $M$-equivalent to $w'$ with respect to $\Gamma$, regardless of whether $\alpha$ and $\beta$ are consecutive or not in $\Gamma$. Hence, $w$ is strongly $M$-equivalent to $w'$.
\end{proof}

Just as $M\!E$-equivalence fails to characterize $M$-equivalence, $M\!S\!E$-equivalance does not characterize strong $M$-equivalence either (see Example~\ref{0306a}). Motivated by the historical development accounted after Definition~\ref{1906a},
our next section is the outcome of our attempt to address the following question.

\begin{question}
How complicated a counterexample witnessing the fact that $M\!S\!E$-equivalence (respectively $M\!E$-equivalence) is strictly weaker than strong $M$-equivalence (respectively $M$-equivalence) has to be?
\end{question}

\section{Some Characterization Results}


Some simple analysis should convince the reader of the following remark.

\begin{remark}\label{1506e}
Suppose $\Sigma=\{a,b,c\}$ and $w\in \Sigma^*$. 
\begin{enumerate}
\item $\vert w\vert_{abc}=0$ if and only if $w=w_1w_2w_3$ for some $w_1\in \{b,c\}^*$, $w_2\in \{a,c\}^*$ and $w_3\in \{a,b\}^*$. 
\item $\vert w\vert_{abc}=1$ if and only if $w=w_1aw_2bw_3c w_4$ for some unique $w_1\in \{b,c\}^*$, $w_2\in c^*$, $w_3\in a^*$ and $w_4\in \{a,b\}^*$. 
\end{enumerate}
\end{remark}

\begin{theorem}\label{2905a}
Suppose $\Sigma=\{a,b,c\}$ and $w, w'\in \Sigma^*$ with $\vert w\vert_{abc}=\vert w'\vert_{abc}\leq 1$. Then $w \overset{s}{\equiv}_M w'$ if and only if 
$w \equiv_{M\!S\!E} w'$.
\end{theorem}

\begin{proof}
By Proposition~\ref{2905b}, it remains to prove the forward direction. Assume $w$ is strongly $M$-equivalent to $w'$. 

First, consider the case $\vert w\vert_{abc}=\vert w'\vert_{abc}=0$.

By Remark~\ref{1506e}, $w=w_1w_2w_3$ and $w'=w_1'w_2'w_3'$ for some $w_1, w_1'\in \{b,c\}^*$, $w_2,w_2'\in \{a,c\}^*$ and $w_3,w_3'\in \{a,b\}^*$.
Notice that $\vert w_2\vert_{ac}=\vert w\vert_{ac} =\vert w'\vert_{ac}=\vert w_2'\vert_{ac}$ and
$\vert w_1\vert_b\vert w_2\vert_{ac}=\vert w\vert_{bac}=\vert w'\vert_{bac}=\vert w_1'\vert_b\vert w_2'\vert_{ac}$. Thus $\vert w_1\vert_b=\vert w_1'\vert_b$ and so $\vert w_3\vert_b=\vert w_3'\vert_b$ as well. 

Since $w \overset{s}{\equiv}_M w'$, by Proposition~\ref{2106a}(3), $\pi_{b,c} (w)\equiv_M \pi_{b,c} (w')$ with respect to $\{b<c\}$, $\pi_{a,c} (w)\equiv_M \pi_{a,c} (w')$
with respect to $\{a<c\}$, and $\pi_{a,b} (w)\equiv_M \pi_{a,b} (w')$ with respect to $\{a<b\}$.

Let $\alpha=\vert w_1\vert_c -\vert w_1'\vert_c$ and $\beta=\vert w_3\vert_a -\vert w_3'\vert_a$. There are altogether four (non-mutually exclusive) cases depending on the nonpositivity or nonnegativity of $\alpha$ and $\beta$. By interchanging $w$ and $w'$, it suffices to consider the following two cases.

\textit{Case 1.} $\alpha\geq 0$ and $\beta \geq 0$.

Since $w_1 \pi_c(w_2)\pi_b(w_3)=\pi_{b,c} (w)\equiv_M \pi_{b,c} (w')= w_1' \pi_c(w_2')\pi_b(w_3') $ with respect to $\{b<c\}$ and $\vert w_3\vert_b=\vert w_3'\vert_b$, it follows that $w_1\equiv_M w_1' \underbrace{c\dotsm c}_{\alpha \text{ times}}$ with respect to $\{b<c\}$. Similarly, $\underbrace{c\dotsm c}_{\alpha \text{ times}}  w_2  \underbrace{a\dotsm a}_{\beta \text{ times}} \equiv_M w_2'$ with respect to $\{a<c\}$
and $w_3 \equiv_M \underbrace{a\dotsm a}_{\beta \text{ times}}  w_3' $ with respect to $\{a<b\}$.
Therefore, by Theorem~\ref{2605c} and Remark~\ref{2106c}, 
$$w_1w_2w_3 \equiv_{M\!S\!E} w_1' \underbrace{c\dotsm c}_{\alpha \text{ times}} w_2w_3 \equiv_{M\!S\!E} w_1'  \underbrace{c\dotsm c}_{\alpha \text{ times}}  w_2
\underbrace{a\dotsm a}_{\beta \text{ times}} w_3'\equiv_{M\!S\!E} w_1'w_2'w_3'.$$

\textit{Case 2.} $\alpha\geq 0$ and $\beta \leq 0$. 

In this case, $\underbrace{c\dotsm c}_{\alpha \text{ times}}  w_2   \equiv_M w_2' \underbrace{a\dotsm a}_{\beta \text{ times}}  $ with respect to $\{a<c\}$
and $ \underbrace{a\dotsm a}_{\beta \text{ times}}  w_3 \equiv_M  w_3' $ with respect to $\{a<b\}$.
Therefore,  $$w_1w_2w_3 \equiv_{M\!S\!E} w_1' \underbrace{c\dotsm c}_{\alpha \text{ times}}  w_2w_3 \equiv_{M\!S\!E}
w_1' w_2'  \underbrace{a\dotsm a}_{\beta \text{ times}}  w_3 \equiv_{M\!S\!E} w_1'w_2'w_3'.$$

Now, consider the case $\vert w\vert_{abc}=\vert w'\vert_{abc}=1$.

By Remark~\ref{1506e},  $w=w_1aw_2bw_3c w_4$ and $w'=w_1'aw_2'bw_3'c w_4'$ for some $w_1, w_1'\in \{b,c\}^*$, $w_2, w_2'\in c^*$, $w_3, w_3'\in a^*$ and $w_4, w_4'\in \{a,b\}^*$.

\begin{claim}
$\vert w_1\vert_b=\vert w_1'\vert_b$, $w_2=w_2'$, $w_3=w_3'$, and $\vert w_4\vert_b=\vert w_4'\vert_b$.
\end{claim}

To prove the claim, without loss of generality,  assume $\vert w_1\vert_b\geq\vert w_1'\vert_b+ 1$.
Note that $\vert w\vert_{bac}=\vert w_1\vert_b \vert aw_2bw_3c\vert_{ac}+\vert w_3\vert=   \vert w_1\vert_b \vert w\vert_{ac}+\vert w_3\vert$.
Hence, $\vert w\vert_{bac} \geq \vert w_1\vert_b \vert w\vert_{ac}  \geq \left( \vert w_1'\vert_b+ 1   \right)\vert w'\vert_{ac} $.
 However, $\vert w'\vert_{ac}>\vert w_3'\vert$. It follows that $\vert w\vert_{bac} > \vert w_1'\vert_b \vert w'\vert_{ac}+\vert w_3'\vert=\vert w'\vert_{bac}$, a contradiction.
Therefore, $\vert w_1\vert_b=\vert w_1'\vert_b$. The rest of the claim follows easily from here.

Since $\pi_{b,c} (w)\equiv_M \pi_{b,c} (w')$ with respect to $\{b<c\}$, using the claim and the right invariance of $M$-equivalence, it follows that $w_1\equiv_M  w_1'$ with respect to $\{b<c\}$. Similarly, $w_4\equiv_M w_4'$ with respect to $\{a<b\}$. Therefore,
$$w_1aw_2bw_3c w_4 \equiv_{M\!S\!E} w_1'aw_2bw_3cw_4 \equiv_{M\!S\!E} w_1'aw_2bw_3cw_4' = w_1'aw_2'bw_3'cw_4'$$
and the proof is complete.
\end{proof}

\begin{example}\label{0306a}
Let $w=bccaabcba$ and $w'=cbabccaab$. It is easy to verify that they are strongly $M$-equivalent. However, Rule $S\!E$ cannot be applied to either of them. Therefore, $w$ and $w'$ are not $M\!S\!E$-equivalent. Since $\vert w\vert_{abc}=2$, this shows that Theorem~\ref{2905a} is optimal.
\end{example}

Similarly, $M\!E$-equivalence can be compared against $M$-equivalence.

\begin{theorem}\label{2705a}
Suppose $\Sigma=\{a<b<c\}$ and $w, w'\in \Sigma^*$ and $\vert w\vert_{abc}=\vert w'\vert_{abc}=0$. Then $w \equiv_M w'$ if and only if 
$w \equiv_{M\!E} w'$.
\end{theorem}

\begin{proof}
It suffices to prove the forward direction as $M\!E$-equivalence immediately implies $M$-equivalence. Assume $w$ is $M$-equivalent to $w'$. By Remark~\ref{1506e}, $w\equiv_1 w_1w_2$ and $w'\equiv_1 w_1'w_2'$, where $w_1,w_1'\in \{b,c\}^*$ and $w_2,w_2'\in \{a,b\}^*$. Since $w\equiv_M w'$, it follows that $\vert w_1 \vert_{bc}=\vert w_1' \vert_{bc}$ and $\vert w_1\vert_{c}=\vert w_1'\vert_c$. Without loss of generality, let $\vert w_1\vert_b - \vert w_1'\vert_b=  \vert w_2'\vert_b-\vert w_2\vert_b=  \alpha \geq 0$.
Then $w_1 \equiv_M w'\underbrace{bb\dotsm b}_{ \alpha \text{ times}  }$ with respect to $\{b<c\}$. Similarly, $\underbrace{bb\dotsm b}_{ \alpha \text{ times}  }  w_2  \equiv_M  w_2'$  with respect to $\{a<b\}$. Therefore, by Theorem~\ref{2605c}, $w \equiv_{M\!E}  w_1w_2  \equiv_{M\!E} w_1'\underbrace{bb\dotsm b}_{ \alpha \text{ times}  }  w_2 \equiv_{M\!E}  w_1' w_2'   \equiv_{M\!E} w'$ as required.
\end{proof}

\begin{example}\label{2006b}
Consider $w=cbbabcab$ and $w'=bcabcbba$. Then $w$ and $w'$ are \mbox{$M$-equivalent} but not $M\!E$-equivalent with respect to $\{a<b<c\}$. These are simpler than the counterexample of length 15 mentioned before. Furthermore, since $\vert w\vert_{abc}=1$, it shows that Theorem~\ref{2705a} is optimal. 
\end{example}

\begin{remark}\label{2206a}
In an older version of this article, it was suggested that \mbox{Example~\ref{2006b}} provides a counterexample of the shortest length. This was confirmed by an anonymous referee, who wrote a program to exhaustively check all the $9841$ ternary words of length at most eight. It was found out that there are $2729$ $M$-equivalence classes compared with $2732$ $M\!E$-equivalence classes. 
\end{remark}

In a certain way, the next theorem allows generation of pairs of $M$-equivalent ternary words that are not $M\!E$-equivalent. In fact, it strongly suggests that when the length of words gets bigger, a pair of $M$-equivalent words are less likely to be $M\!E$-equivalent.

\begin{theorem}\label{2006a}
Suppose $\Sigma=\{a<b<c\}$ and $w, w'\in \Sigma^*$ with \mbox{$\vert w\vert_{abc}=\vert w'\vert_{abc}=1$.} Then $w \equiv_M w'$ if and only if 
$w\equiv_1 w_1abcw_2$ and $w'\equiv_1 w_1'abcw_2'$ for some unique $w_1, w_1'\in \{b,c\}^*$ and $w_2, w_2'\in \{a,b\}^*$ such that
$\Psi_{\{b<c\}}(w_1bc)-\Psi_{\{b<c\}}(w_1'bc)=\begin{pmatrix}
0 & \alpha & 0\\
0 & 0 &  0\\
0 & 0 & 0
\end{pmatrix}       \text{ and }     \Psi_{\{a<b\}}(abw_2)-\Psi_{\{a<b\}}(abw_2')=\begin{pmatrix}
0 & 0 & 0\\
0 & 0 &  -\alpha\\
0 & 0 & 0
\end{pmatrix}$, where $\alpha=\vert w_1\vert_b-\vert w_1'\vert_b$.
Furthermore, $w$ and $w'$ are in fact $M\!E$-equivalent if and only if additionally $\alpha$ is zero.
\end{theorem}

\begin{proof}
By Remark~\ref{1506e}, $w\equiv_1 w_1abcw_2$ and $w'  \equiv_1 w_1'abcw_2'$ for some unique $w_1, w_1'\in \{b,c\}^*$ and $w_2, w_2'\in \{a,b\}^*$.
Then the first conclusion arrives by some simple analysis.

For the second conclusion, if $\alpha$ is zero, then $w_1 \equiv_M w_1'$ with respect to $\{b<c\}$ 
and $w_2 \equiv_M w_2'$ with respect to $\{a<b\}$; hence, $w\equiv_{ME} w'$ by Theorem~\ref{2605c}. 

Conversely, observe that any application of Rule $E2$ on words of the form $v_1av_2bv_3cv_4$, where $v_1\in \{b,c\}^*$, $v_2\in c^*$, $v_3\in a^*$,  and $v_4\in \{a,b\}^*$,  must  be applied either on $v_1$ or on $v_2$. Hence, if $w_1'abcw_2'$ is to be $M\!E$-equivalent to $w_1abcw_2$, it must follow that  $\vert w_1' \vert_b=\vert w_1\vert_b$.
Therefore, if $w$ is  in fact $M\!E$-equivalent to $w'$, then $\alpha=0$.
\end{proof}

\section{Weakly $M$-related}

If strong $M$-equivalence is the universal form of $M$-equivalence, now the existential counterpart will be introduced.

\begin{definition}
Suppose $\Sigma$ is an  alphabet.
Two words $w, w'\in \Sigma^*$ are \emph{weakly $M$-related}, denoted $w \backsim_M w'$,   if{f} $w$ and $w'$ are $M$-equivalent  with respect to some ordered alphabet with underlying alphabet $\Sigma$. 
\end{definition}


Clearly, $ab$ is not weakly $M$-related to $ba$ with respect to $\{a,b\}$ but they are with respect to a (strictly) larger alphabet. In fact, the following is true.

\begin{remark}\label{0506a}
If the alphabet has size at least three, then any swapping of two consecutive distinct letters results in a weakly $M$-related word.
\end{remark}

\begin{example}\label{1705a}
Suppose $\Sigma=\{a,b,c\}$. Then $acb$ and $cab$ are weakly $M$-related. Also,
$cab$ and $cba$ are weakly $M$-related. Assume $acb$ and $cba$ are
\text{$M$-equivalent} with respect to some ordered alphabet $\Gamma$ with underlying alphabet $\Sigma$. Then $a$ and $b$ cannnot be consecutive in $\Gamma$. Similarly, $a$ and $c$ cannnot be consecutive in $\Gamma$. There is no such $\Gamma$.
Therefore, $acb$ and $cba$ cannot be weakly $M$-related.
\end{example}

Example~\ref{1705a} shows that the relation $\backsim_M$ is not transitive. Thus $\backsim_M$ is not an equivalence relation, explaining our choice of ``weakly $M$-related", rather than ``weakly $M$-equivalent". The next theorem says that the transitive closure of $\backsim_M$ is identical to the Parikh equivalence.

\begin{theorem}\label{2405a}
Suppose $\Sigma$ is an alphabet of size at least three and $w, w'\in \Sigma^*$. Then $w$ and $w'$ are equivalent under the transitive closure of $\backsim_M$ if and only if $w$ and $w'$ have the same Parikh vector.
\end{theorem}

\begin{proof}
If $w \backsim_M w'$, then $w$ and $w'$ have the same Parikh vector. By the definition of transitive closure,
the forward direction is immediate. Conversely, assume $w$ and $w'$ have the same Parikh vector. Clearly, $w$ can be transformed into $w'$ by making a sequence of swappings between two consecutive distinct letters. Hence, it suffices to note that each such swapping results in a weakly $M$-related word, and this is true by Remark~\ref{0506a}.
\end{proof}

The relation $\backsim_M$ is not absolute. Any two words having the same Parikh vector become weakly $M$-related simply by expanding their common.
In fact, the following theorem says that there is a bound to the number of auxiliary letters that should be added for that to happen.

\begin{theorem}\label{2501a}
Suppose $w$ and $w'$ have the same Parikh vector. Then $w$ and $w'$ are weakly $M$-related with respect to any alphabet having size at least $2\vert \supp(w)\vert -1$ that includes $\supp(w)$.
\end{theorem}

\begin{proof}
 Suppose $\Sigma$ is any alphabet of size at least $2\vert \supp(w)\vert -1$ that includes $\supp(w)$.
Let $\Gamma$ be any ordered alphabet with underlying alphabet $\Sigma$ such that the letters belonging to $\supp(w)$ are not consecutive in $\Gamma$. Obviously, this is possible because $\vert \Sigma\vert \geq 2\vert \supp(w)\vert -1$. By Theorem~\ref{1206a}, $\Psi_{\Gamma}(w)=\Psi_{\Gamma}(w')$ as all entries
above the second diagonal are zero due to the choice of $\Gamma$ and $\Psi(w)=\Psi(w')$.
Therefore, $w$ and $w'$ are  weakly $M$-related.
\end{proof}

\section{Conclusions}

Strong $M$-equivalence, as highlighted by Theorem~\ref{2701b}, is a natural and interesting notion on its own.
As the characterization of $M$-equivalence for the ternary alphabet has been a decade-old problem, 
it remains to be seen whether strong $M$-equivalence would be as formidable.
However, Theorem~\ref{2905a} says that strong $M$-equivalence can be characterized by $M\!S\!E$-equivalence for the first two ``layers" of ternary words.
It is intriguing which canonical extension of $M\!S\!E$-equivalence may characterize strong $M$-equivalence for the next layer of ternary words. 

Next, it is natural to study the strong version of $M$-ambiguity.
A word is strongly $M$-unambiguous if{f} it is not strongly $M$-equivalent to another distinct word. 
Every $M$-unambiguous word is strongly $M$-unambiguous but not vice versa. 
The characterization of $M$-unambiguous ternary words in the form of a long list was obtained by Serb\v{a}nut\v{a}  in \cite{SS06}. 
Although more ternary words are strongly $M$-unambiguous, the characterization of such words could be given by a shorter list due to the symmetrical nature of strong $M$-equivalence.

Finally, since $\backsim_M$ is not an equivalence relation, together with Theorems~\ref{2405a} and \ref{2501a},
the notion of ``weakly $M$-related" appears to be uninteresting. However, it may lead to other interesting combinatorial questions.

\section*{Acknowledgment}

The notion of ``weakly $M$-related" was suggested by Kiam Heong Kwa. Sound suggestions from an anonymous referee have partially led to the improvement of this article. Furthermore, the computational justification offered in Remark~\ref{2206a} is due to the same referee.
Finally, the author gratefully acknowledges support for this research by a short term grant No. 304/PMATHS/6313077 of Universiti Sains Malaysia.


\begin{thebibliography}{10}

\bibitem{aA07}
A.~Atanasiu, Binary amiable words, {\em Internat. J. Found. Comput. Sci.} {\bf
  18}(2)  (2007)  387--400.

\bibitem{AAP08}
A.~Atanasiu, R.~Atanasiu and I.~Petre, Parikh matrices and amiable words, {\em
  Theoret. Comput. Sci.} {\bf 390}(1)  (2008)  102--109.

\bibitem{aA14}
A.~Atanasiu, Parikh matrix mapping and amiability over a ternary alphabet, {\em Discrete Mathematics and Computer Science. In Memoriam
  Alexandru Mateescu {(1952-2005).}\/},   (2014), pp. 1--12.

\bibitem{AMM02}
A.~Atanasiu, C.~Mart{\'{\i}}n-Vide and A.~Mateescu, On the injectivity of the
  {P}arikh matrix mapping, {\em Fund. Inform.} {\bf 49}(4)  (2002)  289--299.

\bibitem{DS06}
C.~Ding and A.~Salomaa, On some problems of {M}ateescu concerning subword
  occurrences, {\em Fund. Inform.} {\bf 73}  (2006)  65--79.

\bibitem{FR04}
S.~Foss{\'e} and G.~Richomme, Some characterizations of {P}arikh matrix
  equivalent binary words, {\em Inform. Process. Lett.} {\bf 92}(2)  (2004)
  77--82.

\bibitem{MSSY01}
A.~Mateescu, A.~Salomaa, K.~Salomaa and S.~Yu, A sharpening of the {P}arikh
  mapping, {\em Theor. Inform. Appl.} {\bf 35}(6)  (2001)  551--564.

\bibitem{MSY04}
A.~Mateescu, A.~Salomaa and S.~Yu, Subword histories and {P}arikh matrices,
  {\em J. Comput. System Sci.} {\bf 68}(1)  (2004)  1--21.

\bibitem{rP66}
R.~J. Parikh, On context-free languages, {\em J. Assoc. Comput. Mach.} {\bf 13}
   (1966)  570--581.

\bibitem{RS97}
G.~Rozenberg and A.~Salomaa (eds.), {\em Handbook of formal languages. {V}ol.
  1} (Springer-Verlag, Berlin, 1997).

\bibitem{aS05b}
A.~Salomaa, Connections between subwords and certain matrix mappings, {\em
  Theoret. Comput. Sci.} {\bf 340}(2)  (2005)  188--203.

\bibitem{aS05a}
A.~Salomaa, On the injectivity of {P}arikh matrix mappings, {\em Fund. Inform.}
  {\bf 64}  (2005)  391--404.

\bibitem{aS06}
A.~Salomaa, Independence of certain quantities indicating subword occurrences,
  {\em Theoret. Comput. Sci.} {\bf 362}  (2006)  222--231.

\bibitem{aS08}
A.~Salomaa, Subword histories and associated matrices, {\em Theoret. Comput.
  Sci.} {\bf 407}  (2008)  250--257.

\bibitem{aS10}
A.~Salomaa, Criteria for the matrix equivalence of words, {\em Theoret. Comput.
  Sci.} {\bf 411}  (2010)  1818--1827.

\bibitem{vS09}
V.~N. {\c{S}}erb{\u{a}}nu{\c{t}}{\u{a}}, On {P}arikh matrices, ambiguity, and
  prints, {\em Internat. J. Found. Comput. Sci.} {\bf 20}(1)  (2009)  151--165.

\bibitem{SS06}
V.~N. {\c{S}}erb{\u{a}}nu{\c{t}}{\u{a}} and T.~F.
  {\c{S}}erb{\u{a}}nu{\c{t}}{\u{a}}, Injectivity of the {P}arikh matrix
  mappings revisited, {\em Fund. Inform.} {\bf 73}  (2006)  265--283.

\bibitem{wT14}
W.~C. Teh, On core words and the {P}arikh matrix mapping, {\em Internat. J.
  Found. Comput. Sci.} {\bf 26}(1)  (2015)  123--142.

\bibitem{wT15a}
W.~C. Teh, {P}arikh matrices and {P}arikh rewriting systems, {\em axXiv:1506.06476}.

\bibitem{TK14}
W.~C. Teh and K.~H. Kwa, Core words and {P}arikh matrices, {\em Theoret.
  Comput. Sci.} {\bf 582}  (2015)  60--69.




\end{thebibliography}


\end{document}